\documentclass{amsart}
\usepackage[latin1]{inputenc}
\usepackage{amsmath}
\usepackage{amsthm}
\usepackage{amsfonts}
\usepackage{amssymb}
\usepackage[T1]{fontenc}
\usepackage{float}
\usepackage{lscape}
\usepackage[all]{xy}

\theoremstyle{plain}
\newtheorem{theorem}{Theorem}[section]
\newtheorem{prop}[theorem]{Proposition}
\newtheorem{lem}[theorem]{Lemma}
\newtheorem{corol}[theorem]{Corollary}

\theoremstyle{definition}
\newtheorem{defi}[theorem]{Definition}

\newtheorem{rmq}[theorem]{Remark}
\newtheorem{exmp}[theorem]{Example}

\def\modg{{\textrm{-}}{\rm{mod}}}
\def\Gr{{\rm{Gr}}}
\def\Tr{{\rm{Tr}}}

\def\dim{{\rm{dim}\,}}
\def\ddim{{\textbf{dim}\,}}
\def\rep{{\rm{rep}_{\k}}}
\def\qrad{{\rm{q.rad}}}
\def\qsoc{{\rm{q.soc}}}

\def\k{{\bf{k}}}
\def\kQ{\k Q}

\def\LHS{\textrm{(LHS)}\,}
\def\RHS{\textrm{(RHS)}\,}

\def\<{\left<}
\def\>{\right>}

\def\d{{\partial}}

\def\ens#1{\left\{ #1 \right\}}

\def\fl{{\longrightarrow}\,}

\def\CC{{\mathcal{C}}}

\def\A{{\mathbb{A}}}

\def\N{{\mathbb{N}}}
\def\P{{\mathbb{P}}}

\def\Z{{\mathbb{Z}}}

\def\Ext{{\rm{Ext}}}
\def\End{{\rm{End}}}
\def\Hom{{\rm{Hom}}}
\def\Ob{{\rm{Ob}}}

\def\den{{\rm{den}\,}}

\def\im{{\rm{im}\,}}
\def\re{{\rm{re}\,}}

\def\Aaffine{\tilde{\mathbb A}}
\def\Daffine{\tilde{\mathbb D}}
\def\Eaffine{\tilde{\mathbb E}}
\def\Ai{\mathbb A_{\infty}}

\def\X#1#2{X_{R_{#1}^{(#2)}}}
\def\Iff{\Leftrightarrow}

\title{Transverse Quiver Grassmannians and Bases in Affine Cluster Algebras}
\author{G. Dupont}

\begin{document}

\begin{abstract}
	Sherman-Zelevinsky and Cerulli constructed canonically positive bases in cluster algebras associated to affine quivers having at most three vertices. Both constructions involve cluster monomials and normalized Chebyshev polynomials of the first kind evaluated at a certain ``imaginary'' element in the cluster algebra. Using this combinatorial description, it is possible to define for any affine quiver $Q$ a set $\mathcal B(Q)$ which is conjectured to be the canonically positive basis of the acyclic cluster algebra $\mathcal A(Q)$.

	In this article, we provide a geometric realization of the elements in $\mathcal B(Q)$ in terms of the representation theory of $Q$. This is done by introducing an analogue of the Caldero-Chapoton cluster character where the usual quiver Grassmannian is replaced by a constructible subset called \emph{transverse quiver Grassmannian}.
\end{abstract}

\maketitle

\setcounter{tocdepth}{1}
\tableofcontents

\section{Introduction}
	Cluster algebras were introduced by Fomin and Zelevinsky in order to define a combinatorial framework for studying positivity in algebraic groups and canonical bases in quantum groups \cite{cluster1,cluster2,cluster3,cluster4}. Since then, cluster algebras found applications in various areas of mathematics like Lie theory, combinatorics, Teichmüller theory, Poisson geometry or quiver representations.

	A (coefficient-free) cluster algebra $\mathcal A$ is a commutative $\Z$-algebra equipped with a distinguished set of generators, called \emph{cluster variables}, gathered into possibly overlapping sets of fixed cardinality, called \emph{clusters}. Monomials in variables belonging all to the same cluster are called \emph{cluster monomials}. According to the so-called \emph{Laurent phenomenon} \cite{cluster1}, it is known that $\mathcal A$ is a subalgebra of $\Z[\textbf c^{\pm 1}]$ for any cluster $\textbf c$ in $\mathcal A$. A non-zero element $y \in \mathcal A$ is called \emph{positive} if $y$ belongs to $\Z_{\geq 0}[\textbf c^{\pm 1}]$ for any cluster $\textbf c$ in $\mathcal A$. Following \cite{Cerulli:A21}, a $\Z$-basis $\mathcal B \subset \mathcal A$ is called \emph{canonically positive} if the semi-ring of positive elements in $\mathcal A$ coincides with the set of $\Z_{\geq 0}$-linear combinations of elements of $\mathcal B$. Note that if such a basis exists, it is unique.

	The problems of both the existence and the description of a canonically positive basis in an arbitrary cluster algebra are still widely open. Both problems were first solved in the particular case of cluster algebras of finite type $\A_2$ and affine type $\Aaffine_{1,1}$ by Sherman and Zelevinsky \cite{shermanz}. It was later extended by Cerulli for cluster algebras of affine type $\Aaffine_{2,1}$ \cite{Cerulli:A21}. To the best of the author's knowledge, these are the only known constructions of canonically positive bases in cluster algebras.

	Using categorifications of acyclic cluster algebras with cluster categories and cluster characters, it is possible to rephrase Sherman-Zelevinsky and Cerulli constructions in order to place them into the more general context of acyclic cluster algebras associated to arbitrary affine quivers. 

	If $Q$ is an acyclic quiver and $\textbf u$ is a $Q_0$-tuple of indeterminates over $\Z$, we denote by $\mathcal A(Q)$ the acyclic cluster algebra with initial seed $(Q,\textbf u)$. We denote by $\CC_Q$ the associated \emph{cluster category} (over the field $\k$ of complex numbers) and by $X_?:\Ob(\CC_Q) \fl \Z[\textbf u^{\pm 1}]$ the \emph{Caldero-Chapoton map} on $\CC_Q$, also called (canonical) \emph{cluster character} (see Section \ref{section:background} for details). When $Q$ is an affine quiver with positive minimal imaginary root $\delta$, we set 
	$$\mathcal B(Q)=\mathcal M(Q) \sqcup \ens{F_n(X_\delta)X_R|n \geq 1, \, R \textrm{ is a regular rigid $\k Q$-module}}$$
	where $\mathcal M(Q)$ denotes the set of cluster monomials in $\mathcal A(Q)$, $F_n$ denotes the $n$-th \emph{normalized Chebyshev polynomials of the first kind} and $X_{\delta}$ is the evaluation of $X_?$ at any quasi-simple module in a homogeneous tube of the Auslander-Reiten quiver $\Gamma(\kQ\modg)$ of $\kQ$-mod. 

	If $Q$ is of type $\Aaffine_{1,1}$ (respectively $\Aaffine_{2,1}$), the set $\mathcal B(Q)$ coincides with the canonically positive basis constructed in \cite{shermanz} (respectively \cite{Cerulli:A21}). It was conjectured in \cite[Conjecture 7.10]{Dupont:qChebyshev} that, for any affine quiver $Q$, the set $\mathcal B(Q)$ is the canonically positive basis of $\mathcal A(Q)$. Using the so-called \emph{generic basis}, it is possible to prove that, for any affine quiver $Q$, the set $\mathcal B(Q)$ is a $\Z$-basis in $\mathcal A(Q)$ \cite{Dupont:BaseAaffine,DXX:basesv3}. Nevertheless, it is not known if this basis is the canonically positive basis in general.

	An essential problem for investigating this question is due to the fact that the elements of the form $F_n(X_\delta)X_R$ are defined combinatorially and have not yet been given a representation-theoretic or geometric interpretation. The aim of this article is to provide such an interpretation.

	Extending the idea of Caldero and Chapoton \cite{CC}, for any integrable bundle $\mathcal F$ on $\rep(Q)$ (see Section \ref{section:integrable} for definitions), we define a map $\theta_\mathcal F$, called character associated to $\mathcal F$, from the set of objects in $\CC_Q$ to the ring $\Z[\textbf u^{\pm 1}]$. With this terminology, the Caldero-Chapoton map $X_?$ is the character $\theta_{\Gr}$ where $\Gr:M \mapsto \Gr(M)$ denotes the integrable bundle of quiver Grassmannians.

	For any indecomposable $\kQ$-module $M$, we introduce a constructible subset $\Tr(M) \subset \Gr(M)$, called \emph{transverse quiver Grassmannian}. We prove that the bundle $\Tr: M \mapsto \Tr(M)$ is integrable on $\rep(Q)$ and that the elements in $\mathcal B(Q)$ can be described using the associated character $\theta_{\Tr}$. More precisely, we prove that for any $l \geq 1$,
	$$F_l(X_\delta)=\theta_{\Tr}(M)$$
	where $M$ is any indecomposable $\kQ$-module with dimension vector $l \delta$. As opposed to $\theta_{\Gr}$, it turns out that $\theta_{\Tr}$ is independent of the tube containing $M$. In particular, it takes the same values if $M$ belongs to a homogeneous or to an exceptional tube. This is surprising since the usual quiver Grassmannians of two indecomposable modules of dimension $l \delta$ belonging to tubes of different rank are in general completely different.

	Moreover, if $R$ is an indecomposable regular rigid $\kQ$-module, then 
	$$F_l(X_\delta)X_R=\theta_{\Tr}(M)$$
	where $M$ is the unique indecomposable $\kQ$-module of dimension $l\delta+\ddim R$. 

	As a consequence, we obtain the following description of the set $\mathcal B(Q)$~:
	\begin{align*}
		\mathcal B(Q) =
			&\left\{ \theta_{\Tr}(M \oplus R) | M \textrm{ is an indecomposable (or zero) regular $\kQ$-module,}\right.\\
			& \left. \textrm{ $R$ is any rigid object in $\CC_Q$ such that $\Ext^1_{\CC_Q}(M,R)=0$}\right\}.
	\end{align*}

	The paper is organized as follows. In Section \ref{section:background}, we start by recalling several results concerning Chebyshev and generalized Chebyshev polynomials. Then, we recall necessary background on cluster categories and cluster characters associated to acyclic and especially affine quivers. Finally, we recall the known results concerning constructions of bases in affine cluster algebras.

	In Section \ref{section:differences}, we use the combinatorics of generalized Chebyshev polynomials in order to prove relations for cluster characters associated to regular $\kQ$-modules when $Q$ is an affine quiver with minimal imaginary root $\delta$. These relations are generalizations of the so-called \emph{difference property}, introduced previously in \cite{Dupont:BaseAaffine} in order to compute the difference between cluster characters evaluated at indecomposable modules of dimension vector $\delta$ in different tubes.

	Section \ref{section:integrable} introduces the notions of integrable bundles on $\rep(Q)$ and associated characters for any acyclic quiver. In this terminology, the Caldero-Chapoton map is the character associated to the quiver Grassmannian bundle. For affine quivers, we introduce the integrable bundle $\Tr$ of Grassmannian of transverse submodules and see that it coincides with the Caldero-Chapoton map on rigid objects in the cluster category.

	In Section \ref{section:realization}, we prove that the elements in $\mathcal B(Q)$ can be expressed as values of the character $\theta_{\Tr}$ associated to the integrable bundle $\Tr$ of $\rep(Q)$. This provides a geometrization of the set $\mathcal B(Q)$.

	In Section \ref{section:examples}, we illustrate some of our results for quivers of affine types $\Aaffine_{1,1}$ and $\Aaffine_{2,1}$, putting into context the results of \cite{shermanz} and \cite{Cerulli:A21}.

\section{Background, notations and terminology}\label{section:background}
	Given a quiver $Q$, we denote by $Q_0$ its set of arrows and by $Q_1$ its set of vertices. We always assume that $Q_0,Q_1$ are finite sets and that the underlying unoriented graph of $Q$ is connected. A quiver is called \emph{acyclic} if it does not contain any oriented cycles.

	We now fix an acyclic quiver $Q$ and a $Q_0$-tuple $\textbf u=(u_i|i \in Q_0)$ of indeterminates over $\Z$. We denote by $\mathcal A(Q)$ the coefficient-free cluster algebra with initial seed $(Q,\textbf u)$.

	\subsection{Chebyshev polynomials and their generalizations}
		Chebyshev (respectively generalized Chebyshev polynomials) are orthogonal polynomials in one variable (respectively several variables) playing an important role in the context of cluster algebras associated to representation-infinite quivers \cite{shermanz,CZ} (respectively \cite{Dupont:stabletubes,Dupont:qChebyshev}). We recall some basic results concerning these polynomials.
		
		For any $l \geq 0$, the \emph{$l$-th normalized Chebyshev polynomial of the first kind} is the polynomial $F_l$ in $\Z[x]$ defined inductively by
		$$F_0(x)=2, \, F_1(x)=x, $$
		$$F_l(x)=xF_{l-1}(x)-F_{l-2}(x) \textrm{ for any }l \geq 2.$$
		$F_l$ is characterized by the following identity in $\Z[t,t^{-1}]$~:
		$$F_l(t+t^{-1})=t^l+t^{-l}.$$
		These polynomials first appeared in the context of cluster algebras in \cite{shermanz}. 

		For any $l \geq 0$, the \emph{$l$-th (normalized) Chebyshev polynomial of the second kind} is the polynomial $S_l$ in $\Z[x]$ defined inductively by
		$$S_0(x)=1, \, S_1(x)=x, $$
		$$S_l(x)=xS_{l-1}(x)-S_{l-2}(x) \textrm{ for any }l \geq 2.$$
		$S_l$ is characterized by the following identity in $\Z[t,t^{-1}]$~:
		$$S_l(t+t^{-1})=\sum_{k=0}^nt^{n-2k}.$$
		Second kind Chebyshev polynomials first appeared in the context of cluster algebras in \cite{CZ}. For any $l \geq 1$, $S_l(x)$ is the polynomial given by
		$$S_l(x)=\det\left[\begin{array}{rrrrr}
			x & 1 & & & (0)\\
			1 & x & \ddots \\
			& \ddots & \ddots & \ddots \\
			& & \ddots & \ddots & 1\\
			(0) & & & 1 & x
		\end{array}\right]$$
		where the matrix is tridiagonal in $M_l(\Z[x])$.
		First kind and second kind Chebyshev polynomials are related by~:
		$$F_l(x)=S_l(x)-S_{l-2}(x)$$
		for any $l \geq 1$ with the convention that $S_{-1}(x)=0$.

		Fix $\ens{x_i|i\geq 1}$ a family of indeterminates over $\Z$. For any $l \geq 0$, the \emph{$l$-th generalized Chebyshev polynomial} is the polynomial in $\Z[x_1, \ldots, x_l]$ defined inductively by
		$$P_0=1, \, P_1(x_1)=x_1, $$
		$$P_l(x_1, \ldots, x_l)=x_{l}P_{l-1}(x_1,\ldots, x_{l-1})-P_{l-2}(x_1,\ldots, x_{l-2}) \textrm{ for any }l \geq 2.$$
		Equivalently, 
		$$P_l(x_1,\ldots, x_l)=\det\left[\begin{array}{rrrrr}
			x_l & 1 & & & (0)\\
			1 & x_{l-1} & \ddots \\
			& \ddots & \ddots & \ddots \\
			& & \ddots & \ddots & 1\\
			(0) & & & 1 & x_1
		\end{array}\right]$$
		where the matrix is tridiagonal in $M_l(\Z[x_1,\ldots, x_l])$. These polynomials first appeared in the context of cluster algebras in \cite{Dupont:stabletubes} under the name of \emph{generalized Chebyshev polynomials of infinite rank} and similar polynomials also arose in the context of cluster algebras in \cite{YZ:ppalminors,Dupont:qChebyshev}.

	\subsection{Cluster categories and cluster characters}
		Let $\kQ$-mod be the category of finitely generated left-modules over the path algebra $\kQ$ of $Q$. As usual, this category will be identified with the category $\rep(Q)$ of finite dimensional representations of $Q$ over $\k$.
		
		For any vertex $i \in Q_0$, we denote by $S_i$ the simple module associated to $i$, by $P_i$ its projective cover and by $I_i$ its injective hull. We denote by $\<-,-\>$ the homological Euler form defined on $\kQ$-mod by
		$$\<M,N\>=\dim \Hom_{\kQ}(M,N)-\dim \Ext^1_{\kQ}(M,N)$$
		for any two $\kQ$-modules $M,N$. Since $Q$ is acyclic, $\kQ$ is a finite dimensional hereditary algebra so that $\<-,-\>$ is well defined on the Grothendieck group $K_0(\kQ\modg)$. 

		For any $\kQ$-module $M$, the dimension vector of $M$ is
		$$\ddim M=(\dim \Hom_{\kQ}(P_i,M))_{i \in Q_0} \in \N^{Q_0}.$$
		Viewed as a representation of $Q$, $\ddim M=(\dim M(i))_{i \in Q_0}$ where $M(i)$ is the $\k$-vector space at vertex $i$ in the representation $M$ of $Q$. The dimension vector map $\ddim$ induces an isomorphism of abelian groups 
		$$\ddim : K_0(\kQ\modg) \xrightarrow{\sim} \Z^{Q_0}$$
		sending the class of the simple $S_i$ to the $i$-th vector of the canonical basis of $\Z^{Q_0}$. 
		
		The \emph{cluster category} was introduced in \cite{BMRRT} (see also \cite{CCS1} for Dynkin type $\mathbb A$) in order to define a categorical framework for studying the cluster algebra $\mathcal A(Q)$. Let $D^b(\kQ\modg)$ be the bounded derived category of $\kQ$-mod with shift functor $[1]$ and Auslander-Reiten translation $\tau$. The \emph{cluster category} is the orbit category $\CC_Q$ of the auto-functor $\tau^{-1}[1]$ in $D^b(\kQ\modg)$. It is a 2-Calabi-Yau triangulated category. The set of isoclasses of indecomposable objects in $\CC_Q$ can be identified with the union of the set of isoclasses of indecomposable $\kQ$-modules or and the set of isoclasses of shifts of indecomposable projective $\kQ$-modules \cite{K,BMRRT}. In particular, every object $M$ in $\CC_Q$ can be uniquely (up to isomorphism) decomposed into 
		$$M=M_0 \oplus P_M[1]$$
		where $M_0$ is a $\kQ$-module and $P_M$ is a projective $\kQ$-module.

		Given a representation $M$ of $Q$, the \emph{quiver Grassmannian of $M$} is the set $\Gr(M)$ of subrepresentations of $M$. For any element $\textbf e \in \Z^{Q_0}$, the set 
		$$\Gr_{\textbf e}(M)=\ens{N \textrm{ submodule of } M | \ddim N=\textbf e}$$
		is a projective variety. We denote by $\chi(\Gr_{\textbf e}(M))$ its Euler characteristic with respect to the singular cohomology with rational coefficients.

		\begin{defi}[\cite{CC}]
			The \emph{Caldero-Chapoton map} is the map 
			$$X_?:\Ob(\CC_Q) \fl \Z[\textbf u^{\pm 1}]$$ defined by~:
			\begin{enumerate}
				\item[(a)] for any $i \in Q_0$, 
						$$X_{P_i[1]}=u_i~;$$
				\item[(b)] if $M$ is an indecomposable $\kQ$-module, then
					\begin{equation}\label{eq:XM}
						X_M=\sum_{\textbf e \in \N^{Q_0}} \chi(\Gr_{\textbf e}(M))\prod_{i \in Q_0} u_i^{-\<\textbf e, S_i\>-\<S_i,\ddim M-\textbf e\>}~;
					\end{equation}
				\item[(c)] for any two objects $M,N$ in $\CC_Q$, 
						$$X_{M \oplus N}=X_MX_N.$$
			\end{enumerate}
		\end{defi}
		Note that equation (\ref{eq:XM}) also holds for decomposable modules. 

		Caldero and Keller proved that $X_?$ induces a 1-1 correspondence between the set of isoclasses of indecomposable rigid (that is, without self-extensions) objects in $\CC_Q$ and the set of cluster variables in $\mathcal A(Q)$. Moreover, $X_?$ induces a 1-1 correspondence between the set of isoclasses of cluster-tilting objects in $\CC_Q$ and the set of clusters in $\mathcal A(Q)$ \cite[Theorem 4]{CK2}. In particular, we have the following description of cluster monomials in $\mathcal A(Q)$~:
		$$\mathcal M(Q) = \ens{X_M|M \textrm{ is rigid in }\CC_Q}.$$

		For any $\textbf d=(d_i)_{i \in Q_0} \in \Z^{Q_0}$, we set $\textbf u^{\textbf d}=\prod_{i \in Q_0} u_i^{d_i}$. For any Laurent polynomial $L \in \Z[\textbf u^{\pm 1}]$, the \emph{denominator vector} of $L$ it the $Q_0$-tuple $\den(L) \in \Z^{Q_0}$ such that there exists a polynomial $P(u_i|i \in Q_0)$ not divisible by any $u_i$ such that 
		$$L=\frac{P(u_i|i \in Q_0)}{\textbf u^{\den(L)}}.$$
		We define the dimension vector map $\ddim_{\CC_Q}$ on $\CC_Q$ by setting $\ddim_{\CC_Q}M=\ddim M$ if $M$ is a $\kQ$-module and $\ddim_{\CC_Q}P_i[1]=-\ddim S_i$ and extending by additivity. Note that, for any $\kQ$-module $M$, we have $\ddim M=\ddim_{\CC_Q}(M)$ so that, we will abuse notations and write $\ddim M$ for any object in $\CC_Q$.
		Caldero-Keller's denominator theorem \cite[Theorem 3]{CK2} relates the denominator vector of the character with the dimension vector of the corresponding object in the cluster category, namely 
		$$\den(X_M)=\ddim M$$
		for any object $M$ in $\CC_Q$.

	\subsection{Representation theory of affine quivers}
		We shall briefly recall some well-known facts concerning the representation theory of affine quivers. We refer the reader to \cite{SS:volume2,ringel:1099} for details.

		We now fix an affine quiver $Q$, that is, an acyclic quiver of type $\Aaffine_{n},\, (n \geq 1)$, $\Daffine_n,\, (n \geq 4)$, $\Eaffine_n,\, (n =6,7,8)$. We will say that a quiver is of affine type $\Aaffine_{r,s}$ if it is an orientation of an affine diagram of affine type $\Aaffine_{r+s-1}$ with $r$ arrows going clockwise and $s$ arrows going counterclockwise. Let $\mathfrak g_Q$ denote the Kac-Moody algebra associated to $Q$. 
	
		We denote by $\Phi_{>0}$ the set of \emph{positive roots} of $\mathfrak g_Q$, by $\Phi_{>0}^{\re}$ the set of positive \emph{real roots} and by $\Phi_{>0}^{\im}$ the set of positive \emph{imaginary roots}. Since $Q$ is affine, there exists a unique $\delta \in \Phi_{>0}$ such that $\Phi_{>0}^{\im}=\Z_{>0}\delta$.  We always identify the root lattice of $\mathfrak g_Q$ with $\Z^{Q_0}$ by sending the $i$-th simple root of $\mathfrak g_Q$ to the $i$-th vector of the canonical basis of $\Z^{Q_0}$.

		According to Kac's theorem, for any $\textbf d \in \N^{Q_0}$, there exists an indecomposable representation $M$ such that $\ddim M=\textbf d$ if and only if $\textbf d \in \Phi_{>0}$. Moreover, this representation is unique up to isomorphism if and only if $\textbf d \in \Phi_{>0}^{\re}$. A positive root $\textbf d$ is called a \emph{Schur root} if there exists a (necessarily indecomposable) representation $M$ of $Q$ such that $\ddim M=\textbf d$ and $\End_{\kQ}(M) \simeq \k$.
	
		We define a partial order $\leq$ on the root lattice by setting 
		$$\textbf e \leq \textbf f \Leftrightarrow e_i \leq d_i \textrm{ for any } i \in Q_0.$$
		and we set 
		$$\textbf e \lneqq \textbf f\textrm{ if }\textbf e \leq \textbf f\textrm{ and }\textbf e \neq \textbf f.$$

		The Auslander-Reiten quiver $\Gamma(\kQ\modg)$ of $\kQ$-mod contains infinitely many connected components. There exists a connected component containing all the projective (resp. injective) modules, called \emph{preprojective} (resp. \emph{preinjective}) component of $\Gamma(\kQ\modg)$ and denoted by $\mathcal P$ (resp. $\mathcal I$). The other components are called \emph{regular}. A $\kQ$-module $M$ is called \emph{preprojective} (resp. \emph{preinjective}, \emph{regular}) if each indecomposable direct summand of $M$ belongs to a preprojective (resp. preinjective, regular) component. 

		It is convenient to introduce the so-called \emph{defect form} on $\Z^{Q_0}$. It is given by
		$$\d_? : \left\{\begin{array}{rcl}
			\Z^{Q_0} & \fl & \Z \\
			\textbf e & \mapsto & \d_{\textbf e}=\<\delta, \textbf e\>.
		\end{array}\right.$$
		By definition, the defect $\d_M$ of a $\kQ$-module $M$ is the defect $\d_{\ddim M}$ of its dimension vector. It is well-known that an indecomposable $\kQ$-module $M$ is preprojective (resp. preinjective, regular) if and only if $\d_M < 0$ (resp. $>0$, $=0$). 

		The regular components in $\Gamma(\kQ\modg)$ form a $\P^1(\k)$-family of tubes. Thus, for every tube $\mathcal T$, there exists an integer $p \geq 1$, called \emph{rank} of $\mathcal T$ such that $\mathcal T \simeq \Z\Ai/(\tau^p)$. The tubes of rank one are called \emph{homogeneous}, the tubes of rank $p>1$ are called \emph{exceptional}. At most three tubes are exceptional in $\Gamma(\kQ\modg)$. It is well-known that the full subcategory of $\kQ$-mod formed by the objects in any tube $\mathcal T$ is standard, that is, isomorphic to the mesh category of $\mathcal T$. It is also known that there are neither morphisms nor extensions between pairwise distinct tubes.

		An indecomposable regular $\kQ$-module $M$ is called \emph{quasi-simple} if it is at the mouth of the tube, or equivalently, if it does not contain any proper regular submodule. For any quasi-simple module $R$ in a tube $\mathcal T$ and any integer $l \geq 1$, we denote by $R^{(l)}$ the unique indecomposable $\kQ$-module with quasi-socle $R$ and quasi-length $l$. For any indecomposable regular $\kQ$-module $R^{(l)}$, we denote by 
		$$\qsoc R^{(l)}=R$$
		the \emph{quasi-socle} of $M$ and by 
		$$\qrad R^{(l)} = R^{(l-1)}$$ the quasi-radical of $M$ with the convention that $R^{(0)}=0$.

		For any indecomposable regular $\kQ$-module $M$, we have the following~:
		$$M \textrm{ is rigid } \Leftrightarrow \ddim M \lneqq \delta~;$$
		$$\End_{\kQ}(M) \simeq \k \Leftrightarrow \ddim M \leq \delta.$$

		Cluster characters associated to modules in tubes are known to be governed by the combinatorics of generalized Chebyshev polynomials. More precisely, it is proved in \cite[Theorem 5.1]{Dupont:stabletubes} that for any quasi-simple module $M$ in a tube $\mathcal T$, we have 
		$$X_{M^{(l)}}=P_l(X_M, X_{\tau^{-1}M}, \ldots, X_{\tau^{-l+1}M}).$$
		In particular, if $\mathcal T$ is homogeneous, we get 
		$$X_{M^{(l)}}=S_l(X_M)$$
		recovering a result of \cite{CZ}.
		
		In \cite[Theorem 7.2]{Dupont:stabletubes}, generalized Chebyshev polynomials provide multiplication formulas for cluster characters associated to indecomposable regular $\kQ$-modules. The following theorem will be essential in the proofs~:
		\begin{theorem}[\cite{Dupont:stabletubes}]\label{theorem:multintubes}
			Let $Q$ be an affine quiver and $\mathcal T$ be a tube of rank $p$ in $\Gamma(\kQ\modg)$. Let $R_i, i \in \Z$ denote the quasi-simple modules in $\mathcal T$ ordered such that $\tau R_i \simeq R_{i-1}$ and $R_{i+p} \simeq R_i$ for any $i \in \Z$. Let $m,n>0$ be integers and $j \in [0,p-1]$. Then, for every $k \in \Z$ such that $0 < j+kp  \leq n$ and $m \geq n-j-kp$, we have the following identity~:
			$$X_{R_{j}^{(m)}}X_{R_0^{(n)}}=X_{R_0^{(m+j+kp)}}X_{R_{j}^{(n-j-kp)}}+X_{R_0^{(j+kp-1)}}X_{R_{n+1}^{(m+j+kp-n-1)}}.$$
		\end{theorem}

	\subsection{Bases in affine cluster algebras}
		We shall now review some results concerning the construction of $\Z$-bases in cluster algebras associated to affine quivers. In this section, $Q$ still denotes an affine quiver with positive minimal imaginary root $\delta$.
	
		It is known that if $M,N$ are quasi-simple modules in distinct homogeneous tubes, then $X_M=X_N$ (see e.g. \cite{Dupont:BaseAaffine}). We denote by $X_{\delta}$ this common value. Following the terminology of \cite{Dupont:BaseAaffine}, $X_{\delta}$ is called the \emph{generic variable of dimension $\delta$}.

		The following holds~:
		\begin{theorem}[\cite{Dupont:BaseAaffine,DXX:basesv3}]\label{theorem:genericbasis}
			Let $Q$ be an affine quiver. Then 
			$$\mathcal G(Q)=\mathcal M(Q) \sqcup \ens{X_\delta^lX_R|l \geq 1, R \textrm{ is a regular rigid $\kQ$-module}}$$
			is a $\Z$-basis of $\mathcal A(Q)$.

			Moreover, $\den$ induces a 1-1 correspondence from $\mathcal G(Q)$ to $\Z^{Q_0}$.
		\end{theorem}
		The set $\mathcal G(Q)$ is called the \emph{generic basis} of $\mathcal A(Q)$. 

		Since $F_l$ and $S_l$ are monic polynomials of degree $l$, it follows that, for any affine quiver $Q$, the sets
		$$\mathcal B(Q)=\mathcal M(Q) \sqcup \ens{F_l(X_\delta)X_R|l \geq 1, R \textrm{ is a regular rigid $\kQ$-module}}$$
		and
		$$\mathcal C(Q)=\mathcal M(Q) \sqcup \ens{S_l(X_\delta)X_R|l \geq 1, R \textrm{ is a regular rigid $\kQ$-module}}$$
		are $\Z$-bases of the cluster algebra $\mathcal A(Q)$. 

		When $Q$ is the Kronecker quiver $\mathcal B(Q)$ coincides with the canonically positive basis constructed by Sherman and Zelevinsky \cite{shermanz} and $\mathcal C(Q)$ coincides with the basis constructed by Caldero and Zelevinsky \cite{CZ}. When $Q$ is a quiver of affine type $\Aaffine_{2,1}$, the basis $\mathcal B(Q)$ is the canonically positive basis of $\mathcal A(Q)$ constructed by Cerulli \cite{Cerulli:A21}.

		Since $X_{\delta}=X_M$ for any quasi-simple module in a homogeneous tube, it follows that $S_l(X_\delta)=X_{M^{(l)}}$ so that the set $\mathcal C(Q)$ has an interpretation in terms of the cluster character $X_?$. No such interpretation is known for the set $\mathcal B(Q)$. The aim of this paper is to provide one.
		
		The map $\phi: \mathcal G(Q) \fl \mathcal B(Q)$ preserving cluster monomials and sending $X_\delta^l X_R$ to $F_l(X_\delta)X_R$ for any $l \geq 1$ and any rigid regular module $R$ is a 1-1 correspondence. We denote by 
		$$\mathfrak b_?:\left\{\begin{array}{rcl}
			\Z^{Q_0} & \xrightarrow{1:1} & \mathcal B(Q)\\
			\textbf d & \mapsto & \mathfrak b_{\textbf d}
		\end{array}\right.$$
		the 1-1 correspondence obtained by composing the above bijection with the one provided in Theorem \ref{theorem:genericbasis}.

\section{Difference properties of higher orders}\label{section:differences}
	In this section, $Q$ still denotes an affine quiver with positive minimal imaginary root $\delta$.

	\subsection{The difference property}
		In \cite{Dupont:BaseAaffine} was introduced the \emph{difference property} which relates the possibly different values of cluster characters evaluated at different indecomposable representations of dimension $\delta$. This difference property is crucial in \cite{Dupont:BaseAaffine}. It is also an essential ingredient in the present article since the transverse Grassmannians will precisely arise from difference properties of higher orders.

		This difference property was established in \cite{Dupont:BaseAaffine} for affine type $\Aaffine$ and in \cite{DXX:basesv3} in general. It can be expressed as follows~:
		\begin{theorem}[\cite{Dupont:BaseAaffine, DXX:basesv3}]\label{theorem:diffppt}
			Let $Q$ be an affine quiver and $M$ be any indecomposable module of dimension $\delta$, then
			$$\mathfrak b_{\delta}=X_\delta=X_M-X_{\qrad M/\qsoc M}$$
			with the convention that $X_{\qrad M/\qsoc M}=0$ if $M$ is quasi-simple.
		\end{theorem}

	\subsection{Higher difference properties}
		The aim of this section is to provide an analogue of Theorem \ref{theorem:diffppt} for $\mathfrak b_{\textbf d}$ when $\textbf d$ is any positive root with zero defect. We will first consider the imaginary roots and then the real roots of defect zero.

		We fix a tube $\mathcal T$ in $\Gamma(\kQ\modg)$ of rank $p \geq 1$. The quasi-simples of $\mathcal T$ are denoted by $R_i$, with $i \in \Z/p\Z$, ordered such that $\tau R_i \simeq R_{i-1}$ for any $i \in \Z/p\Z$. Note that for any $l \geq 1, 0 \leq k \leq p-1$ and any $i \in \Z/p\Z$, $\ddim R_i^{(lp)} =l\delta \in \Phi_{>0}^{\im}$ and $\ddim R_i^{(lp+k)} \in \Phi_{>0}^{\re}$ if $k \neq 0$.
		
		The following technical lemma will be used in the proof of Proposition \ref{prop:diffppthomogene}~:
		\begin{lem}\label{lem:keyidentity}
			With the above notations, for any $l \geq 1$
			$$X_{R_0^{(lp-1)}}X_{R_1^{(p-1)}}=X_{R_0^{(p-1)}}X_{R_1^{(lp-1)}}$$
		\end{lem}
		\begin{proof}
			We first notice that generalized Chebyshev polynomials are symmetric in the sense that for every $i \in \Z$ and $n \geq 1$, 
			$$P_n(x_{i}, \ldots, x_{i+n-1})=P_n(x_{i+n-1}, \ldots, x_{i}).$$
			If $l=1$, the result is obvious. We thus fix some $l \geq 2$. For technical convenience, we denote by $R_i, i \in \Z$ the quasi-simple modules in $\mathcal T$ and we assume that $R_i \simeq R_{i+p}$ for every $i \in \Z$.
			Consider the morphism of $\Z$-algebras 
			$$\phi:\left\{\begin{array}{rcll} 
				    \Z[X_{R_0}, \ldots, X_{R_{lp-1}}] & \fl & \Z[X_{R_0}, \ldots, X_{R_{lp-1}}]\\
				    X_{R_i} & \mapsto & X_{R_{lp-1-i}} & \textrm{ for all }i=0, \ldots, lp-1.\\
			\end{array}\right.$$
			It is well defined since $X_{R_0}, \ldots, X_{R_{p-1}}$ are known to be algebraically independent over $\Z$ (see e.g. \cite{Dupont:stabletubes}).
			
			According to Theorem \ref{theorem:multintubes}, we have 
			$$\X{1}{p-1}\X{0}{lp-1}=\X{0}{p}\X{1}{lp-2}-\X{p+1}{(l-1)p-2}.$$
			According to \cite[Theorem 5.1]{Dupont:stabletubes}, each of the $X_{R_j}^{(k)}$ appearing above lies in $\Z[X_{R_0}, \ldots, X_{R_{lp-1}}]$. We can thus apply $\phi$ and we get 
			\begin{equation}\label{eq:equationphi}
				\phi(\X{1}{p-1})\phi(\X{0}{lp-1})=\phi(\X{0}{p})\phi(\X{1}{lp-2})-\phi(\X{p+1}{(l-1)p-2}).
			\end{equation}

			We now compute these images under $\phi$.
			\begin{align*}
				\phi(\X{1}{p-1})
					&= \phi(P_{p-1}(X_{R_1}, \ldots, X_{R_{p-1}}))\\
					&= P_{p-1}(\phi(X_{R_1}), \ldots, \phi(X_{R_{p-1}}))\\
					&= P_{p-1}(X_{R_{lp-2}}, \ldots, X_{R_{(l-1)p}})\\
					&= P_{p-1}(X_{R_{(l-1)p}}, \ldots, X_{R_{lp-2}})\\
					&= X_{R_{(l-1)p}^{(p-1)}}\\
					&= X_{R_{0}^{(p-1)}}\\
			\end{align*}
			\begin{align*}
				\phi(\X{0}{lp-1})
					&= \phi(P_{lp-1}(X_{R_0}, \ldots, X_{R_{lp-2}}))\\
					&= P_{lp-1}(\phi(X_{R_0}), \ldots, \phi(X_{R_{lp-2}}))\\
					&= P_{lp-1}(X_{R_{lp-1}}, \ldots, X_{R_{1}})\\
					&= P_{lp-1}(X_{R_{1}}, \ldots, X_{R_{lp-1}})\\
					&= X_{R_{1}^{(lp-1)}}\\
			\end{align*}
			\begin{align*}
				\phi(\X{0}{p})
					&= \phi(P_{p}(X_{R_0}, \ldots, X_{R_{p-1}}))\\
					&= P_{p}(\phi(X_{R_0}), \ldots, \phi(X_{R_{p-1}}))\\
					&= P_{p}(X_{R_{lp-1}}, \ldots, X_{R_{(l-1)p}})\\
					&= P_{p}(X_{R_{(l-1)p}}, \ldots, X_{R_{lp-1}})\\
					&= X_{R_{(l-1)p}^{(p)}}\\
					&= X_{R_{0}^{(p)}}\\
			\end{align*}
			\begin{align*}
				\phi(\X{1}{lp-2})
					&= \phi(P_{lp-2}(X_{R_1}, \ldots, X_{R_{lp-2}}))\\
					&= P_{lp-2}(\phi(X_{R_1}), \ldots, \phi(X_{R_{lp-2}}))\\
					&= P_{lp-2}(X_{R_{lp-2}}, \ldots, X_{R_{1}})\\
					&= P_{lp-2}(X_{R_{1}}, \ldots, X_{R_{lp-2}})\\
					&= X_{R_{1}^{(lp-2)}}\\
			\end{align*}
			\begin{align*}
				\phi(\X{p+1}{(l-1)p-2})
					&= \phi(P_{(l-1)p-2}(X_{R_{p+1}}, \ldots, X_{R_{lp-2}}))\\
					&= P_{(l-1)p-2}(\phi(X_{R_{p+1}}), \ldots, \phi(X_{R_{lp-2}}))\\
					&= P_{(l-1)p-2}(X_{R_{(l-1)p-2}}, \ldots, X_{R_{1}})\\
					&= P_{(l-1)p-2}(X_{R_{1}}, \ldots, X_{R_{(l-1)p-2}})\\
					&= X_{R_{1}^{((l-1)p-2})}\\
			\end{align*}
			Replacing in Equation (\ref{eq:equationphi}), we get
			\begin{align*}
			\X{0}{p-1}\X{1}{lp-1}
				&=\X{0}{p}\X{1}{lp-2}-\X{1}{(l-1)p-2}\\
				&=\X{0}{p}\X{1}{lp-2}-\X{p+1}{(l-1)p-2}\\
				&=\X{1}{p-1}\X{0}{lp-1}
			\end{align*}
			which proves the lemma.
		\end{proof}
	
		We can now prove some higher difference properties for imaginary roots.
		\begin{prop}\label{prop:diffppthomogene}
			Fix $l \geq 1$. Then for any indecomposable representation $M$ in $\rep(Q,l \delta)$, we have 
			$$\mathfrak b_{l \delta}=F_l(X_\delta)=X_M-X_{\qrad M/ \qsoc M}$$
			with the convention that $X_{\qrad M/ \qsoc M}=0$ if $M$ is quasi-simple.
		\end{prop}
		\begin{proof}
			We first treat the case where $M$ is an indecomposable representation of dimension $l \delta$ in a homogeneous tube. It is not necessary to prove it separately but in this particular case, the proof is straightforward. We write $M=R^{(l)}$ for some quasi-simple module $R$ in a homogeneous tube. If $l=1$, the proposition follows from Theorem \ref{theorem:diffppt}. If $l \geq 2$, $\qrad M/\qsoc M \simeq R^{(l-2)}$ so that 
			\begin{align*}
				X_M - X_{\qrad M/\qsoc M}
					& = X_{R^{(l)}} - X_{R^{(l-2)}}\\
					& = S_l(X_{R})-S_{l-2}(X_{R}) \\
					& = F_l(X_{R}) \\
					& = F_l(X_{\delta}).
			\end{align*}

			We now assume that $M$ is an indecomposable representation of dimension $l \delta$ in an exceptional tube $\mathcal T$ of rank $p  \geq 2$. We denote $R_0, \ldots, R_{p-1}$ the quasi-simples in $\mathcal T$ ordered such that $\tau R_i \simeq R_{i-1}$ for any $i \in \Z/p\Z$. We can thus write $M \simeq R_i^{(lp)}$ for some $i \in \Z/p\Z$. Without loss of generality, we assume that $i=0$. In order to simplify notations, for any $l \geq 1$, we denote by 
			$$\Delta_{l}=X_{R_0^{(lp)}}-X_{R_1^{(lp-2)}}.$$
			We thus have to prove that for any $l \geq 1$,
			$$\Delta_{l}=F_l(X_{\delta}).$$
			The central tool in this proof is Theorem \ref{theorem:multintubes}. According to Theorem \ref{theorem:diffppt}, we have 
			$$X_\delta=\X{0}{p}-\X{1}{p-2}$$
			so that the proposition holds for $l=1$. 
			
			We now prove it for $l=2$. We have
			\begin{align*}
				F_2(X_\delta) 
					&= X_\delta^2-2 \\
					&= (\X{0}{p}-\X{1}{p-2})^2-2\\
					&= \X{0}{p}^2+\X{1}{p-2}^2-2 \X{0}{p}\X{1}{p-2}-2\\
			\end{align*}
			but according to the almost split multiplication formula (\cite[Proposition 3.10]{CC}), we have
			$$\X{0}{p-1}\X{1}{p-1}=\X{0}{p}\X{1}{p-2}$$
			so that 
			$$F_2(X_\delta) = \X{0}{p}^2-2 \X{0}{p-1}\X{1}{p-1}+\X{1}{p-2}^2.$$
			But, according to Theorem \ref{theorem:multintubes}, we have
			\begin{align*}
				\X{0}{p}^2
					&=\X{0}{p}\X{0}{p}\\
					&=\X{0}{2p}+\X{0}{p-1}\X{1}{p-1}
			\end{align*}
			so that finally
			$$F_2(X_\delta) = \X{0}{2p}-\X{0}{p-1}\X{1}{p-1}+\X{1}{p-2}^2.$$
			
			Thus, 
			\begin{align*}
				& F_2(X_\delta) = \Delta_2 \\
				\Iff & -\X{1}{2p-2} = \X{1}{p-2}^2 -\X{0}{p-1}\X{1}{p-1}\\
				\Iff & \X{1}{2p-2} + \X{1}{p-2}^2 -\X{0}{p-1}\X{1}{p-1} = 0.
			\end{align*}
			But
			$$\X{1}{2p-2}=-\X{1}{p-3}\X{0}{p-1}+\X{1}{p-2}\X{p-1}{p}$$
			So that, 
			\begin{align*}
				& F_2(X_\delta) = \Delta_2 \\
				\Iff & \X{1}{p-2}^2+\X{1}{p-2}\X{p-1}{p}-\X{0}{p-1}\X{1}{p-1}-\X{0}{p-1}\X{1}{p-3} = 0\\
				\Iff & \X{1}{p-2} \left[ \X{1}{p-2}+\X{p-1}{p}\right] - \X{0}{p-1} \left[ \X{1}{p-3}+\X{1}{p-1}\right] = 0.
			\end{align*}
			But Theorem \ref{theorem:multintubes} gives
			$$X_{R_{p-1}}\X{0}{p-1}=\X{1}{p-2}+\X{p-1}{p}$$
			$$\X{1}{p-2}X_{R_{p-1}}=\X{1}{p-3}+\X{1}{p-1}$$
			so that we get 
			$$F_2(X_\delta) = \Delta_2$$
			and the proposition is proved for $l=2$.
		
			For $l >2$, we will use the three term relations for first kind Chebyshev polynomials
			$$F_{l}(x)=xF_{l-1}(x)-F_{l-2}(x).$$ 
			Thus, it is enough to prove that for any $l \geq 2$, 
			$$\Delta_{l+1}=\Delta_1 \Delta_l - \Delta_{l-1}.$$
			In order to simplify notations, we denote by \LHS the left-hand side of the equality and by \RHS the right-hand side of the above equality. We thus have
			\begin{align*}
				\RHS 
					&= (\X{0}{p}-\X{1}{p-2})(\X{0}{lp}-\X{1}{lp-2})-(\X{0}{(l-1)p}-\X{1}{(l-1)p-2})\\
					&= \X 0 p \X 0 {lp} - \X 0 p \X 1 {lp-2} - \X 1 {p-2} \X 0 {lp} \\
					& + \X 1 {p-2} \X 1 {lp -2} - \X 0 {(l-1)p} + \X 1 {(l-1)p -2}.
			\end{align*}
			But, according to the multiplication theorem, we get
			$$\X 0 p \X 0 {lp}=\X 0 {(l+1)p} + \X 0 {p-1} \X 1 {lp -1}$$
			So that 
			\begin{align*}
			\LHS = \RHS & \Leftrightarrow \X 0 {p-1} \X 1 {lp-1} - \X 0 p \X 1 {lp-2} - \X 1 {p-2} \X 0 {lp}\\
				& + \X 1 {p-2} \X 1 {lp-2} - \X 0 {(l-1)p} + \X 1 {(l-1)p-2}+\X 1 {(l+1)p-2} = 0.
			\end{align*}
			Applying the multiplication theorem, we get
			$$\X 0 {lp-2} \X {lp-2} p = \X 0 {(l+1)p-2}+\X 0 {lp-3} \X {lp-1} {p-1}$$
			so that 
			$$\X 1 {(l+1)p-2} = \X {1} {lp -2} \X {p-1} {p} - \X{1}{lp-3} \X{0}{p-1}$$
			and thus
			\begin{align*}
			\LHS = \RHS 
				& \Leftrightarrow \X 0 {p-1} \X 1 {lp-1} - \X 0 p \X 1 {lp-2} \\
				& - \X 1 {p-2} \X 0 {lp} + \X 1 {p-2} \X 1 {lp-2} \\
				& - \X 0 {(l-1)p} + \X 1 {(l-1)p-2}\\
				& + \X {1} {lp -2} \X {p-1} {p} - \X{1}{lp-3} \X{0}{p-1} = 0
			\end{align*}
			but 
			$$\X 0 p \X 1 {lp-2} = \X 0 {lp -1} \X 1 {p-1}+ \X 1 {(l-1)p-2}.$$
			Thus,
			\begin{align*}
			\LHS = \RHS 
				& \Leftrightarrow \X 0 {p-1} \X 1 {lp-1} - \X 0 {lp-1} \X 1 {p-1} \\
				& - \X 1 {p-2} \X 0 {lp} + \X 1 {p-2} \X 1 {lp-2} \\
				&- \X 0 {(l-1)p} + \X 1 {lp-2} \X {p-1} p - \X 1 {lp-3} \X 0 {p-1} =0\\
				& \Leftrightarrow \X 0 {p-1} \X 1 {lp-1} - \X 0 {lp-1} \X 1 {p-1} \\
				& - \X 1 {p-2} \X 0 {lp} - \X 0 {(l-1)p} - \X 1 {lp-3} \X 0 {p-1} \\
				& + \X 1 {lp-2} \left( \X 1 {p-2} + \X {p-1} p \right) = 0.
			\end{align*}
			Theorem \ref{theorem:multintubes} gives
			$$X_{R_{p-1}}\X 0 {p-1} =\X 1 {p-2} + \X {p-1} p$$
			so that 
			\begin{align*}
			\LHS = \RHS 
				& \Leftrightarrow \X 0 {p-1} \left( \X 1 {lp-1} - \X 1 {lp-3}+\X 1 {lp-2} X_{R_{p-1}}\right)\\
				& - \X 0 {(l-1)p} - \X 0 {lp-1} \X 1 {p-1} - \X 1 {p-2} \X 0 {lp} =0
			\end{align*}
			The three term relation for generalized Chebyshev polynomials gives
			$$\X 1 {lp-1}=X_{R_{p-1}}\X 1 {lp-2} - \X 1 {lp-3}.$$
			Thus
			\begin{align*}
			\LHS = \RHS 
				& \Leftrightarrow 2 \X 0 {p-1} \X 1 {lp-1} \\
				& - \left( \X 0 {(l-1)p} + \X 0 {lp-1} \X 1 {p-1} + \X 0 {lp} \X 1 {p-2}\right) = 0
			\end{align*}
			Theorem \ref{theorem:multintubes} gives
			$$\X 0 {p-1}\X 1 {lp-1}=\X 0 {lp} \X 1 {p-2} + \X 0 {(l-1)p}$$
			so that we finally get 
			\begin{align*}
			\LHS = \RHS 
				& \Leftrightarrow \X 0 {p-1} \X 1 {lp-1}- \X 0 {lp-1} \X 1 {p-1} = 0.
			\end{align*}
			The second equality holds by Lemma \ref{lem:keyidentity} so that we proved that for any $l \geq 2$,
			$$\Delta_{l+1}=\Delta_1 \Delta_l - \Delta_{l-1}.$$
			Since we know that $\Delta_1=X_\delta$ and $\Delta_2=F_2(X_\delta)$, it follows that $\Delta_l=F_l(X_\delta)$ for any $l \geq 1$. This finishes the proof.
		\end{proof}

		We are now able to prove the general difference property~:
		\begin{theorem}\label{theorem:generaldifferenceppty}
			Let $Q$ be an affine quiver, $\mathcal T$ be a tube of rank $p \geq 1$ in $\Gamma(\kQ\modg)$. Then for any $l \geq 1$ and any $0 \leq k \leq p-1$, we have
			$$\mathfrak b_{l \delta + \ddim R_0^{(k)}}=\X 0 k F_l(X_\delta)=\X 0 {lp+k} - \X {k+1} {lp-k-2}$$
			with the convention that $\X {0}{-1}=0$.
		\end{theorem}
		\begin{proof}
			The first equality follows from the fact that 
			$$\den(\X 0 k X_{\delta}^l)=\ddim R_0^{(k)}+l \delta$$
			so that
			$$\mathfrak b_{l \delta + \ddim R_0^{(k)}}=\X 0 k F_l(X_\delta).$$
			We now prove that
			$$\X 0 k F_l(X_\delta)=\X 0 {lp+k} - \X {k+1} {lp-k-2}$$
			with the convention that $\X {0}{-1}=0$.

			We denote by \LHS the left-hand side and by \RHS the right-hand side of the above equation.
			\begin{align*}
				\LHS
					& = \X 0 k \left( \X k {lp} - \X {k+1} {lp-2}\right)\\
					& = \X 0 k \X k {lp} - \X 0 k\X {k+1} {lp-2}\\
					& = \X 0 {lp+k} + \X 0 {k-1} \X {k+1} {lp-1} -\X 0 k\X {k+1} {lp-2}.
			\end{align*}
			If $l=1$ and $k=p-1$, we get 
			\begin{align*}
				\LHS
					& = \X 0 {lp+k} + \X 0 {p-2} \X {0} {p-1} -\X 0 {p-1}\X {0} {p-2}. \\
					& = \X 0 {lp+k} \\
					& = \RHS.
			\end{align*}
	
			Otherwise, $\LHS = \RHS$ if and only if 
			\begin{equation}\label{eq:generaldifferenceppty}
				\X {k+1} {lp-k-2} = \X 0 k \X {k+1} {lp-2} - \X 0 {k-1} \X {k+1} {lp-1}
			\end{equation}
			holds.

			Using the three term recurrence relations for generalized Chebyshev polynomials, we have 
			$$\X 0 k = X_{R_{k-1}}\X 0 {k-1} - \X 0 {k-2}$$
			and 
			$$\X {k+1} {lp-1} = X_{R_{lp+k-1}} \X {k+1} {lp-2} - \X {k+1} {lp-3}$$
			so that, replacing in the right-hand side of equality (\ref{eq:generaldifferenceppty}), we get~:
			\begin{align*}
				\X 0 k \X {k+1} {lp-2} - \X 0 {k-1} \X {k+1} {lp-1}
					 & = \X 0 {k-1} \X {k+1} {lp-3} - \X 0 {k-2}
			\end{align*}
			Thus, by induction, we get 
			$$\X 0 k \X {k+1} {lp-2} - \X 0 {k-1} \X {k+1} {lp-1}= X_{R_0} \X {k+1} {lp-k-1} - \X {k+1} {lp-k}.$$
			Now, the three term recurrence relation gives 
			\begin{align*}
				\X {k+1} {lp-k} 
					& = X_{R_{k+1+lp-k-1}}\X {k+1} {lp-k-1} - \X {k+1} {lp-k-2}\\
					& = X_{R_0}\X {k+1} {lp-k-1} - \X {k+1} {lp-k-2}\\
			\end{align*}
			and thus 
			$$X_{R_0} \X {k+1} {lp-k-1} - \X {k+1} {lp-k}=\X {k+1} {lp-k-2}$$
			so that equality (\ref{eq:generaldifferenceppty}) holds.
		\end{proof}

		As a corollary, for any positive root $\textbf d$ with defect zero, we obtain a description of $\mathfrak b_{\textbf d}$ as a certain difference of cluster characters~:
		\begin{corol}
			Let $Q$ be an affine quiver, $\textbf d$ be a positive root with defect zero. Let $M$ be any indecomposable representation of dimension $\textbf d$. Then, there exists a quasi-simple module $R_0$ in a tube of rank $p \geq 1$, an integer $0 \leq k \leq p-1$ and an integer $l \geq 0$ such that $\textbf d=l\delta+\ddim R_0^{(k)}$. Moreover, for any such $R_0,k,l$, we have 
			$$\mathfrak b_{\textbf d}=\X 0 k F_l(X_\delta)=\X 0 {lp+k} - \X {k+1} {lp-k-2}.$$
			where $R_i$, with $i \in \Z/p\Z$, are the quasi-simple modules in $\mathcal T$ ordered such that $\tau R_i \simeq R_{i-1}$ for every $i \in \Z/p\Z$.
		\end{corol}

\section{Integrable bundles on $\rep(Q)$ and their characters}\label{section:integrable}
	In the previous section, we obtained a realization of the elements $\mathfrak b_{\textbf d}$ associated to defect zero roots as differences of cluster characters. The aim of this section is to introduce a new map $\theta_{\Tr}$ such that these elements correspond precisely to values of $\theta_{\Tr}$.

	Unless it is otherwise specified, $Q$ denotes an arbitrary acyclic quiver in this section.

	\subsection{Integrable bundles}
		For any $\textbf d \in \N^{Q_0}$, the representation variety $\rep(Q,\textbf d)$ of dimension $\textbf d$ is the set of all representations $M$ of $Q$ with dimension vector $\textbf d$. Note that 
		$$\rep(Q,\textbf d) \simeq \prod_{i \fl j \in Q_1} \Hom_{\k}(\k^{d_i},\k^{d_j})$$
		so that $\rep(Q,\textbf d)$ is an affine irreducible variety.
		
		\begin{defi}
			Let $Q$ be any acyclic quiver. An \emph{integrable bundle} on $\rep(Q)$ is a map 
			$$\mathcal F:M \mapsto \mathcal F(M) \subset \Gr(M)$$
			defined on the set of indecomposable objects in $\rep(Q)$ such that for any $M \in \rep(Q)$ the following hold~:
			\begin{enumerate}
				\item for any $\textbf e \in \N^{Q_0}$, $\mathcal F_{\textbf e}(M) = \mathcal F(M) \cap \Gr_{\textbf e}(M)$ is constructible~;
				\item if $M \simeq N$ in $\rep(Q)$, then $\chi(\mathcal F_{\textbf e}(M)) \simeq \chi(\mathcal F_{\textbf e}(N))$ for any $\textbf e \in \N^{Q_0}$.
			\end{enumerate}
		\end{defi}

		\begin{rmq}
			Note that, if $\mathcal F$ is an integrable bundle on $\rep(Q)$, then the family $(\chi(\mathcal F_{\textbf e}(M)))_{\textbf e \in \N^{Q_0}}$ has finite support.
		\end{rmq}

		\begin{exmp}
			The map $M \mapsto \Gr(M)$ is an integrable bundle called \emph{quiver Grassmannian bundle}. 
		\end{exmp}

		For any $\kQ$-module $M$ and any submodule $U \subset M$, we set
		$$\Gr^U(M)=\ens{N \in \Gr(M) |U \textrm{ is a submodule of }N}.$$
		This is a constructible subset in the quiver Grassmannian $\Gr(M)$.

		If $Q$ is an affine quiver, we define another integrable bundle $\Tr$ as follows. Let $M$ be an indecomposable $\kQ$-module. If $M$ is rigid, we set $\Tr(M)=\Gr(M)$. If $M$ is not rigid, it is regular and we can thus write $M=R_0^{(lp+k)}$ for some quasi-simple module $R_0$ in a tube of rank $p \geq 1$, $l \geq 1$ and $0 \leq k \leq p-1$. There exists a non-zero monomorphism $\iota : R_0^{(lp-1)} \fl R_0^{(lp)}$ such that $\Hom_{\kQ}(R_0^{(lp-1)}, R_0^{(lp)}) \simeq \k \iota$. The set $\iota(\Gr^{R_0^{(k+1)}}(R_0^{(lp-1)}))$ is a constructible subset of $\Gr(R_0^{(lp)})$ and since $\Hom_{\kQ}(R_0^{(lp-1)}, R_0^{(lp)}) \simeq \k \iota$, it does not depend on the choice of $\iota$. We can thus identify $\Gr^{R_0^{(k+1)}}(R_0^{(lp-1)})$ with a constructible subset of $\Gr(R_0^{(lp)})$. With these notations and identifications, we set 
		$$\Tr(M)=\Gr(M) \setminus \Gr^{R_0^{(k+1)}}(R_0^{(lp-1)}).$$
		Note that if $l=0$, $M$ is rigid and we recover the equality $\Tr(M)=\Gr(M)$.  

		For every dimension vector $\textbf e \in \N^{Q_0}$ and any indecomposable $\kQ$-module $M$, the \emph{transverse quiver Grassmannian of $M$} (of dimension $\textbf e$) is the constructible subset of $\Gr_{\textbf e}(M)$~:
		$$\Tr_{\textbf e}(M)=\ens{N \in \Tr(M)| \ddim N = \textbf e}.$$
		The map
		$$\Tr:M \mapsto \Tr(M) \subset \Gr(M)$$
		is an integrable bundle on $\rep(Q)$.

	\subsection{Character associated to an integrable bundle}
		Extending an idea of Caldero and Chapoton, we associate to any integrable bundle on $\rep(Q)$ a map from the set of objects in $\CC_Q$ to the ring $\Z[\textbf u^{\pm 1}]$ of Laurent polynomials in the initial cluster of $\mathcal A(Q)$.

		\begin{defi}
			Let $\mathcal F$ be an integrable bundle on $\rep(Q)$. The \emph{character associated to $\mathcal F$} is the map 
			$$\theta_{\mathcal F}(?):\Ob(\CC_Q) \fl \Z[\textbf u^{\pm 1}]$$
			given by
			\begin{enumerate}
				\item[(a)] If $M \simeq P_i[1]$ for some $i \in Q_0$, then $\theta_{\mathcal F}(P_i[1]) = u_i$~;
				\item[(b)] If $M$ is an indecomposable $\kQ$-module, then 
					$$\theta_{\mathcal F}(M) = \sum_{\textbf e \in \N^{Q_0}} \chi(\mathcal F_{\textbf e}(M)) \prod_{i \in Q_0} u_i^{-\<\textbf e, S_i\>-\<S_i,\ddim M-\textbf e\>}~;$$
				\item[(c)] $\theta_{\mathcal F}(M \oplus N)=\theta_{\mathcal F}(M)\theta_{\mathcal F}(N)$ for any two objects $M,N$ in $\CC_Q$.
			\end{enumerate}
 		\end{defi}

		We now prove that $\theta_{\Tr}$ coincides with $X_?$ on the set of rigid objects in $\CC_Q$. In particular, this will allow to realize cluster monomials in terms of $\theta_{\Tr}$.
		\begin{lem}\label{lem:clustermonomials}
			Let $Q$ be an affine quiver. Then, for any rigid object $M$ in $\CC_Q$, we have $\theta_{\Tr}(M)=X_M$.
			In particular, 
			$$\mathcal M(Q)=\ens{\theta_{\Tr}(M)|M \textrm{ is rigid in }\CC_Q}.$$
		\end{lem}
		\begin{proof}
			Let $M$ be a rigid object in $\CC_Q$. We write $M=P_{i_1}[1] \oplus \cdots \oplus P_{i_r}[1] \oplus M_1 \oplus \cdots \oplus M_s$ where each $P_{i_j}$ is an indecomposable projective $\kQ$-module and each $M_i$ is an indecomposable module. Moreover, since $M$ is rigid, each $M_i$ is a rigid $\kQ$-module and thus $\Tr(M_i)=\Gr(M_i)$ for any $i \in \ens{1, \ldots, s}$. In particular, it follows that $\theta_{\Tr}(M_i)=X_{M_i}$ for any $i \in \ens{1, \ldots, s}$. Then, 
			\begin{align*}
				\theta_{\Tr}(M)
					& = \theta_{\Tr}(P_{i_1}[1] \oplus \cdots \oplus P_{i_r}[1] \oplus M_1 \oplus \cdots \oplus M_s) \\
					& = \theta_{\Tr}(P_{i_1}[1]) \cdots \theta_{\Tr}(P_{i_r}[1]) \theta_{\Tr}(M_1) \cdots \theta_{\Tr}(M_s) \\
					& = u_{i_1} \cdots u_{i_r} X_{M_1} \cdots X_{M_s}\\
					& = X_{P_{i_1}[1] \oplus \cdots \oplus P_{i_r}[1] \oplus M_1 \oplus \cdots \oplus M_s}\\
					& = X_M
			\end{align*}
			The second assertion follows directly from Caldero-Keller's realization of cluster monomials~:
			\begin{align*}
				\mathcal M(Q) 
					&=\ens{X_M|M \textrm{ is rigid in }\CC_Q}\\
					&=\ens{\theta_{\Tr}(M)|M \textrm{ is rigid in }\CC_Q}.
			\end{align*}
		\end{proof}

\section{A geometrization of $\mathcal B(Q)$}\label{section:realization}
	We now relate the transverse character with the difference properties obtained in Section \ref{section:differences}. This will provide a realization of the elements in $\mathcal B(Q)$ in terms of $\theta_{\Tr}$.

	\subsection{From difference properties to $\theta_{\Tr}$}
		Using Theorem \ref{theorem:generaldifferenceppty}, we first deduce a realization in terms of $\theta_{\Tr}$ of the elements in $\mathcal B(Q)$ corresponding to positive roots with zero defect~:
		\begin{theorem}\label{theorem:zerodefect}
			Let $\textbf d$ be any positive root. Then
			$$\mathfrak b_{\textbf d}=\theta_{\Tr}(M)$$
			where $M$ is any indecomposable representation of dimension $\textbf d$.
		\end{theorem}
		\begin{proof}
			If $\textbf d$ is a positive root with non-zero defect, then $\textbf d$ is real and there exists a unique indecomposable representation $M$ in $\rep(Q,\textbf d)$. Moreover, this representation has to be preprojective or preinjective. In both cases, it is rigid and thus $\mathfrak b_{\textbf d}=X_M=\theta_{\Tr}(M)$. We can thus assume that $\textbf d \in \N^{Q_0}$ is a root with zero defect.
			
			Let $M$ be an indecomposable representation in $\rep(Q,\textbf d)$. It is necessarily contained in a tube $\mathcal T$ of rank $p \geq 1$. We denote by $R_i$, with $i \in \Z/p\Z$ the quasi-simple modules in $\mathcal T$ ordered such that $\tau R_i \simeq R_{i-1}$ for any $i \in \Z/p\Z$. We can write $\textbf d = l \delta + \textbf n$ where $\textbf n$ is either a real Schur root or zero. 
			If $\textbf n\neq 0$, there exists a unique indecomposable representation $N$ in $\rep(Q,\textbf n)$. In any case, if $M \simeq R_0^{(lp+k)}$ with $l \geq 0$ and $0 \leq k \leq p-1$, $N$ is the rigid representation $R_0^{(k)}$ (still with the convention that $R_0^{(0)}=0$) and
			$$\mathfrak b_{\textbf d}=X_{R_0^{(k)}} F_l(X_\delta).$$
			Now according to Theorem \ref{theorem:generaldifferenceppty}, we have 
			$$X_{R_0^{(k)}} F_l(X_\delta)=\X 0 {lp+k} - \X {k+1} {lp-k-2}$$
			For any $\textbf e \in \N^{Q_0}$, the map 
			$$\left\{\begin{array}{rcl}
					\Gr^{R_0^{(k+1)}}_{\textbf e}(R_0^{(lp-1)}) & \fl & \Gr_{\textbf e - \ddim R_0^{(k+1)}}(R_{k+1}^{(lp-k-2)})\\
					U & \mapsto & U/R_0^{(k+1)}
			\end{array}\right.$$
			is an algebraic isomorphism and we denote by $c_{\textbf e} \in \Z$ the common value of the Euler characteristics of these constructible sets.
			Fix now some $\textbf e \in \N^{Q_0}$, the monomial corresponding to $\textbf e$ in $\X 0 {lp+k}$ is
			$$c_{\textbf e}\prod_i u_i^{-\< \textbf e, S_i\>-\<S_i,\ddim R_0^{(lp+k)} - \textbf e\>}$$
			and the monomial corresponding to $\textbf e - \ddim R_0^{(k+1)}$ in $\X {k+1} {lp-k-2}$ is 
			$$c_{\textbf e}\prod_i u_i^{-\< \textbf e-\ddim R_0^{(k+1)}, S_i\>-\<S_i,\ddim R_{k+1}^{(lp-k-2)}+\ddim R_0^{(k+1)} - \textbf e\>}.$$
			We now prove that these monomials are the same.
			For any $i=0, \ldots, p-1$, we set by $r_i=\ddim R_i$ and we denote by $c$ the Coxeter transformation on $\Z^{Q_0}$ induced by the Auslander-Reiten translation. We recall that for any $\beta, \gamma \in \Z^{Q_0}$, we have $\<\gamma, c(\beta) \>=-\<\beta, \gamma\>$.
			With these notations, we have 
			\begin{align*}
				\ddim R_0^{(k+1)}	 & =r_0 + \cdots + r_k\\
				\ddim R_{k+1}^{(lp-k-2)} & =(l-1)\delta + r_{k+1} + \cdots r_{p-2}
			\end{align*}
			so that 
			$$\ddim R_0^{(k+1)} + \ddim R_{k+1}^{(lp-k-2)}= l\delta - r_{p-1}.$$
			We now compute the exponents~:
			\begin{align*}
				& -\< \textbf e, S_i\>-\<S_i,\ddim R_0^{(lp+k)} - \textbf e\>\\
				& = -\<\textbf e, S_i\>-\<S_i, l\delta+r_0+\cdots + r_{k-1}-\textbf e\>\\
				& = -\<\textbf e, S_i\>-\<S_i, l\delta-\textbf e\>-\<S_i,r_0+\cdots + r_{k-1}\>\\
				& = -\<\textbf e, S_i\>-\<S_i, l\delta-\textbf e\>-\<r_1+\cdots + r_{k},S_i\>
			\end{align*}
			and
			\begin{align*}
				& -\< \textbf e-\ddim R_0^{(k+1)}, S_i\>-\<S_i,\ddim R_{k+1}^{(lp-k-2)}+\ddim R_0^{(k+1)} - \textbf e\>\\
				& = -\< \textbf e, S_i\>+\<r_0+\cdots+r_k, S_i\>+\<S_i,r_{p-1}\>-\<S_i,l \delta-\textbf e\>\\
				& = -\< \textbf e, S_i\>-\<S_i,l \delta-\textbf e\>+\<r_1+\cdots+r_k, S_i\>
			\end{align*}
			so that the two monomials are the same. Thus, 
			\begin{align*}
				& \X 0 {lp+k} - \X {k+1} {lp-k-2}\\
					& = \sum_{\textbf e} \chi(\Gr_{\textbf e}(R_0^{(lp+k)})) \prod_i u_i^{ -\<\textbf e, S_i\>-\<S_i, l\delta-\textbf e\>-\<r_1+\cdots + r_{k},S_i\>}\\
					& - \sum_{\textbf e} \chi(\Gr_{\textbf e-\ddim R_0^{(k+1)}}(R_{k+1}^{(lp-k-2)})) \prod_i u_i^{ -\<\textbf e, S_i\>-\<S_i, l\delta-\textbf e\>-\<r_1+\cdots + r_{k},S_i\>}\\
					& = \sum_{\textbf e} \chi \left(\Gr_{\textbf e}(R_0^{(lp+k)}) \setminus \Gr_{\textbf e}^{R_0^{(k+1)}}(R_0^{(lp-1)})\right) \prod_i u_i^{ -\<\textbf e, S_i\>-\<S_i, l\delta-\textbf e\>-\<r_1+\cdots + r_{k},S_i\>}\\
					& = \sum_{\textbf e} \chi(\Tr_{\textbf e}(R_0^{(lp+k)})) \prod_i u_i^{ -\<\textbf e, S_i\>-\<S_i, \ddim R_0^{(lp+k)}-\textbf e\>}.
			\end{align*}
			This finishes the proof.
		\end{proof}

	\subsection{Realization of $\mathcal B(Q)$ in terms of $\theta_{\Tr}$}
		Summing up the previous results, we deduce the following geometric description of $\mathcal B(Q)$~:
		\begin{theorem}\label{theorem:BQ}
			Let $Q$ be an affine quiver, then 
			\begin{align*}
				\mathcal B(Q) =
				&\left\{ \theta_{\Tr}(M \oplus R) | M \textrm{ is an indecomposable (or zero) regular $\kQ$-module,}\right.\\
				& \left. \textrm{ $R$ is any rigid object in $\CC_Q$ such that $\Ext^1_{\CC_Q}(M,R)=0$}\right\}.	
			\end{align*}
		\end{theorem}
		\begin{proof}
			We denote by $\mathcal S$ the right-hand-side of the claimed equality. 
			By definition, we have 
			$$\mathcal B(Q)=\mathcal M(Q) \sqcup \ens{F_l(X_\delta)X_R|l \geq 1, R \textrm{ is a regular rigid $\kQ$-module}}.$$
			We first prove that every $\mathcal S \subset \mathcal B(Q)$. Let $R$ be a rigid object in $\CC_Q$, then $\theta_{\Tr}(R)$ is a cluster monomial by Lemma \ref{lem:clustermonomials}. Fix now $M$ to be an indecomposable regular $\kQ$-module in a tube $\mathcal T$ such that $\Ext^1_{\CC_Q}(M,R)=0$. If $M$ is rigid, then $M \oplus R$ is rigid in $\CC_Q$ and $\theta_{\Tr}(M)\theta_{\Tr}(R)=\theta_{\Tr}(M \oplus R)$ is a cluster monomial by Lemma \ref{lem:clustermonomials}. Now, if $M$ is non-rigid, then $\textbf d =\ddim M$ is a positive root of defect zero, thus, $\theta_{\Tr}(M)=\mathfrak b_{\textbf d}$ by Theorem \ref{theorem:zerodefect}. Thus, there exists $l \geq 1$ and $N$ a indecomposable rigid (or zero) module in $\mathcal T$ such that $\textbf d=l\delta+\ddim N$. According to Theorem \ref{theorem:zerodefect}, we have 
			\begin{align*}
				\theta_{\Tr}(M \oplus R)
					& = \theta_{\Tr}(M) \theta_{\Tr}(R)\\
					& = \mathfrak b_{l\delta+\ddim N}\theta_{\Tr}(R)\\
					& = F_l(X_\delta)X_N\theta_{\Tr}(R)\\
					& = F_l(X_\delta)\theta_{\Tr}(N)\theta_{\Tr}(R)\\
					& = F_l(X_\delta)\theta_{\Tr}(N \oplus R).
			\end{align*}
			Since $\Ext^1_{\CC_Q}(M,R)=0$, we have $\Ext^1_{\kQ}(M,R)=0$ and $\Ext^1_{\kQ}(R,M)=0$. Thus, it follows easily that $\Ext^1_{\kQ}(N,R)=0$ and $\Ext^1_{\kQ}(R,N)=0$ so that $N \oplus R$ is a rigid regular $\kQ$-module. In particular, $\theta_{\Tr}(N \oplus R)=X_{N \oplus R}$ and thus 
			$$\theta_{\Tr}(M \oplus R)=F_l(X_\delta)X_{N \oplus R} \in \mathcal B(Q).$$

			Conversely, fix an elements in $\mathcal B(Q)$. If $x$ is a cluster monomial, then according to Lemma \ref{lem:clustermonomials}, there exists some rigid object $M$ in $\CC_Q$ such that $x=\theta_{\Tr}(M)$. Thus, $x \in \mathcal S$. Fix now some regular rigid $\kQ$-module $R$ and some integer $l \geq 1$. Then the direct summands of $R$ belong to exceptional tubes. We fix an indecomposable $\kQ$-module $M$ of dimension vector $l \delta$ in a homogeneous tube. Then, $\Ext^1_{\CC_Q}(M,R)=0$. According to Theorem \ref{theorem:zerodefect}, we have $F_l(X_\delta)X_R=\theta_{\Tr}(M)X_R$ but $R$ is rigid so that $X_R=\theta_{\Tr}(R)$. Thus, 
			$$F_l(X_\delta)X_R=\theta_{\Tr}(M)\theta_{\Tr}(R)=\theta_{\Tr}(M \oplus R) \in \mathcal S.$$
			This finishes the proof.
		\end{proof}

\section{Examples}\label{section:examples}
	We shall now study two examples corresponding to cases where it is known that $\mathcal B(Q)$ is the canonically positive basis in $\mathcal A(Q)$.
	
	\subsection{The $\Aaffine_{1,1}$ case}
		Let $Q$ be the Kronecker quiver, that is, the affine quiver of type $\Aaffine_{1,1}$ with the following orientation~:
		$$\xymatrix{
			Q: & 1 \ar@<+2pt>[r]\ar@<-2pt>[r] & 2 
		}$$
		with minimal imaginary root $\delta=(11)$.

		For any $\lambda \in \k$, we set
		$$\xymatrix{
			M_\lambda: & \k \ar@<+2pt>[r]^1 \ar@<-2pt>[r]_\lambda & \k
		}$$
		and 
		$$\xymatrix{
			M_\infty: & \k \ar@<+2pt>[r]^0 \ar@<-2pt>[r]_1 & \k.
		}$$
		It is well-known that every tube in $\Gamma(\kQ\modg)$ is homogeneous and that the family $\ens{M_\lambda|\lambda \in \k \sqcup \ens{\infty}}$ is a complete set of representatives of pairwise non-isomorphic quasi-simple $\kQ$-modules. 

		For any $n \geq 1$, the indecomposable representations of quasi-length $n$ are given by 
		$$\xymatrix{
			M_\lambda^{(n)}: & \k^n \ar@<+2pt>[r]^1 \ar@<-2pt>[r]_{J_n(\lambda)} & \k^n
		}$$
		for any $\lambda \in \k$ and 
		$$\xymatrix{
			M_\infty^{(n)}: & \k^n \ar@<+2pt>[r]^{J_n(0)} \ar@<-2pt>[r]_1 & \k^n
		}$$
		where $J_n(\lambda) \in M_n(\k)$ denotes the Jordan block of size $n$ associated to the eigenvalue $\lambda$. Quiver Grassmannians and transverse quiver Grassmannians of indecomposable representations with quasi-length 2 are described in Figure \ref{figure:TrA11} below. 

		Note that $\kQ$-mod contains no regular rigid modules. It follows that in this case
		$$\mathcal B(Q)=\mathcal M(Q) \sqcup \ens{\theta_{\Tr}(M)|M \textrm{ is an indecomposable regular $\kQ$-module}}.$$
		According to \cite{shermanz}, this set is the canonically positive basis of $\mathcal A(Q)$.

		\begin{landscape}
			\begin{figure}
				$$\begin{array}{|r||cc|cc|cc||c|}
					\hline
					\textbf e 
						& \Gr_{\textbf e}(M_0^{(2)}) & \Tr_{\textbf e}(M_0^{(2)}) 
						& \Gr_{\textbf e}(M_\lambda^{(2)}) & \Tr_{\textbf e}(M_\lambda^{(2)}) 
						& \Gr_{\textbf e}(M_\infty^{(2)}) & \Tr_{\textbf e}(M_\infty^{(2)})
						& \textbf u^{\<-\textbf e,S_i\>-\<S_i,2\delta-\textbf e\>}\\
					\hline
					(00)	& \ens 0 & \ens 0 
						& \ens 0 & \ens 0 
						& \ens 0 & \ens 0 
						& \displaystyle \frac{u_1^2}{u_2^2}\\ &&&&&&&\\
					(01)	& \P^1 \times \ens{S_2} & \P^1 \times \ens{S_2}
						& \P^1 \times \ens{S_2} & \P^1 \times \ens{S_2}
						& \P^1 \times \ens{S_2} & \P^1 \times \ens{S_2}
						& \displaystyle \frac{1}{u_2^2}\\ &&&&&&&\\
					(02)	& \ens {S_2\oplus S_2} & \ens {S_2\oplus S_2}
						& \ens {S_2\oplus S_2} & \ens {S_2\oplus S_2}
						& \ens {S_2\oplus S_2} & \ens {S_2\oplus S_2}
						& \displaystyle \frac{1}{u_1^2u_2^2}\\ &&&&&&&\\
					(11)	& \ens {M_0} & \emptyset
						& \ens {M_\lambda} & \emptyset
						& \ens {M_\infty} & \emptyset
						& \displaystyle 1 \\ &&&&&&&\\
					(12)	& \ens {P_1,M_0 \oplus S_2} & \ens {P_1,M_0 \oplus S_2}
						& \P^1 \times \ens{M_\lambda \oplus S_2} & \P^1 \times \ens{M_\lambda \oplus S_2}
						& \ens {P_1,M_\infty \oplus S_2} & \ens {P_1,M_\infty \oplus S_2}
						& \displaystyle \frac{1}{u_1^2}\\ &&&&&&&\\
					(22)	& \ens {M_0^{(2)}} & \ens {M_0^{(2)}}
						& \ens {M_\lambda^{(2)}} & \ens {M_\lambda^{(2)}}
						& \ens {M_\infty^{(2)}} & \ens {M_\infty^{(2)}}
						& \displaystyle \frac{u_2^2}{u_1^2}\\ &&&&&&&\\
					\hline
				\end{array}$$
				\caption{Grassmannians and transverse Grassmannians of indecomposable modules of quasi-length 2 in type $\Aaffine_{1,1}$}\label{figure:TrA11}
			\end{figure}
		\end{landscape}
		From Figure \ref{figure:TrA11}, we see that for any $\lambda \in \k \sqcup \ens{\infty}$, 
		$$
			\theta_{\Tr}(M_\lambda^{(2)})
				=\theta_{\Gr}(M_\lambda^{(2)})-1
				= X_{M_\lambda^{(2)}}-1
				= S_2(X_{M_\lambda})-1
				= F_2(X_{M_\lambda})
				= \mathfrak b_{2\delta}$$
		This illustrates Theorem \ref{theorem:zerodefect}.

	\subsection{The $\Aaffine_{2,1}$ case}
		We now consider the quiver $Q$ of affine type $\Aaffine_{2,1}$ equipped with the following orientation~:
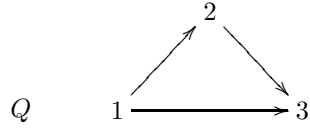
		$$\xymatrix{
			&& 2 \ar[rd] \\
			Q & 1 \ar[rr] \ar[ru] && 3
		}$$
		The minimal imaginary root of $Q$ is $\delta=(111)$.
		For any $\lambda \in \k$, we set 
		$$
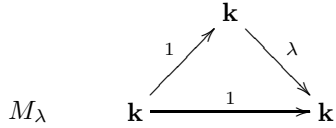
\xymatrix{
			&& \k \ar[rd]^{\lambda} \\
			M_\lambda & \k \ar[rr]^1 \ar[ru]^1 && \k
		}$$
		and 
		
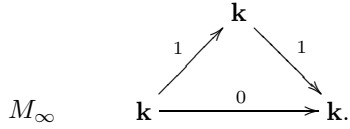
$$\xymatrix{
			&& \k \ar[rd]^{1} \\
			M_\infty & \k \ar[rr]^0 \ar[ru]^1 && \k.
		}$$
		$\Gamma(\kQ\modg)$ contains exactly one exceptional tube $\mathcal T$ of rank 2 whose quasi-simples are 
		$$
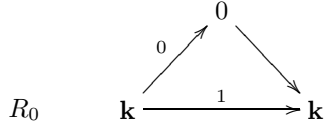
\xymatrix{
			&& 0 \ar[rd] \\
			R_0 & \k \ar[rr]^1 \ar[ru]^0 && \k
		}$$
		and 
		
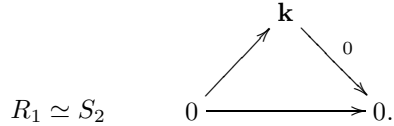
$$\xymatrix{
			&& \k \ar[rd]^0 \\
			R_1 \simeq S_2 & 0 \ar[rr] \ar[ru] && 0.
		}$$

		The set $\ens{M_\lambda | \lambda \in \k \sqcup \ens{\infty}} \sqcup \ens{R_0^{(2)}}$ is a complete set of representatives of pairwise non-isomorphic indecomposable representations in $\rep(Q,\delta)$. For any $\lambda \neq 0,\infty$, $M_\lambda$ is a quasi-simple $\kQ$-module in a homogeneous tube. Moreover, $M_0=R_1^{(2)}$ and $M_\infty$ is quasi-simple in a homogeneous tube.

		Quiver Grassmannians and transverse quiver Grassmannians of indecomposable representations of dimension $\delta$ are described in Figure \ref{figure:TrA21} below. For simplicity, we only listed the dimension vectors corresponding giving non-empty quiver Grassmannians.

		\begin{landscape}
			\begin{figure}
				$$\begin{array}{|r||cc|cc|cc|cc||c|}
					\hline
					\textbf e 
						& \Gr_{\textbf e}(M_\lambda) & \Tr_{\textbf e}(M_\lambda)
						& \Gr_{\textbf e}(M_0) & \Tr_{\textbf e}(M_0)
						& \Gr_{\textbf e}(R_0^{(2)}) & \Tr_{\textbf e}(R_0^{(2)})
						& \Gr_{\textbf e}(M_\infty) & \Tr_{\textbf e}(M_\infty)
						& \textbf u^{\<-\textbf e,S_i\>-\<S_i,\delta-\textbf e\>}\\
					\hline
					(000)	& \ens 0 & \ens 0 
						& \ens 0 & \ens 0 
						& \ens 0 & \ens 0 
						& \ens 0 & \ens 0 
						& \displaystyle \frac{u_1}{u_3}\\ &&&&&&&&&\\
					(001)	& \ens{S_3} & \ens{S_3}
						& \ens{S_3} & \ens{S_3}
						& \ens{S_3} & \ens{S_3}
						& \ens{S_3} & \ens{S_3}
						& \displaystyle \frac{1}{u_2u_3}\\ &&&&&&&&&\\
					(010)	& \emptyset & \emptyset
						& \ens {S_2} & \emptyset
						& \emptyset & \emptyset
						& \emptyset & \emptyset
						& \displaystyle 1\\ &&&&&&&&&\\
					(011)	& \ens {P_2} & \ens {P_2}
						& \ens{S_2 \oplus S_3} & \ens{S_2 \oplus S_3}
						& \ens {P_2} & \ens {P_2}
						& \ens {P_2} & \ens {P_2}
						& \displaystyle \frac{1}{u_1u_2}\\ &&&&&&&&&\\
					(101)	& \emptyset & \emptyset
						& \emptyset & \emptyset
						& \ens {R_0} & \emptyset
						& \emptyset & \emptyset
						& \displaystyle 1\\ &&&&&&&&&\\
					(111)	& \ens {M_\lambda} & \ens {M_\lambda}
						& \ens{M_0} & \ens{M_0}
						& \ens{R_0^{(2)}} & \ens{R_0^{(2)}}
						& \ens {M_\infty} & \ens {M_\infty}
						& \displaystyle \frac{u_3}{u_1}\\ &&&&&&&&&\\
					\hline
				\end{array}$$
				\caption{Grassmannians and transverse Grassmannians for quasi-length 2 in type $\Aaffine_{2,1}$}\label{figure:TrA21}
			\end{figure}
		\end{landscape}

		In Figure \ref{figure:TrA21}, we observe that $X_{M_\lambda}=X_{M_0}-1=X_{M_\infty}-1$, illustrating Theorem \ref{theorem:diffppt}. Also, we see that $\theta_{\Tr}(M_\lambda)=\theta_{\Tr}(M_0)=\theta_{\Tr}(M_\infty)$ for any $\lambda \in \k \setminus \ens 0$ so that the transverse character does not depend on the chosen tube. Moreover, 
		$$\theta_{\Tr}(M_\lambda)=X_{M_\lambda}=F_1(X_\delta)$$
		illustrating Theorem \ref{theorem:zerodefect}.

		\begin{rmq}
			Figure \ref{figure:TrA21} justifies the terminology ``transverse submodule''. Indeed, we see that, given two indecomposable regular modules $M$ and $N$ having the same dimension vectors, the submodules $U$ in $\Tr(M)$ are those having a corresponding submodule in $\Gr(N)$. In some sense, we can see $U$ as a submodule ``common'' to $M$ and $N$. This is why we call it \emph{transverse}. 

			As suggested by Bernhard Keller, this notion of transversality should have a more precise meaning in the context of deformation theory. Some connections are known at this time, this should be discussed in a forthcoming article.
		\end{rmq}

\section*{Acknowledgements}
	This paper was written while the author was at the university of Sherbrooke as a CRM-ISM postdoctoral fellow under the supervision of the Ibrahim Assem, Thomas Br\"ustle and Virginie Charette. He would like to thank Giovanni Cerulli Irelli for motivating the investigation of higher difference properties during his stay at the University of Padova in june 2009. This was the starting point of this work. He would also like to thank the rest of the algebra group of Padova for their kind hospitality. Finally, he would like to thank Bernhard Keller, Philippe Caldero and Frédéric Chapoton for interesting discussions on the topic.


\newcommand{\etalchar}[1]{$^{#1}$}

\end{document}